\documentclass[11pt]{amsart}

\usepackage{amssymb,amscd,amsthm,amsxtra}
\usepackage{dsfont,latexsym}
\usepackage{mathrsfs}

\usepackage{color}
\definecolor{citation}{rgb}{0.2,0.58,0.2} 
\definecolor{formula}{rgb}{0.1,0.2,0.6}
\definecolor{url}{rgb}{0.3,0,0.5} 

\usepackage[colorlinks=true,linkcolor=formula,urlcolor=url,citecolor=citation]{hyperref}  

\newtheorem{thm}{Theorem}[section]
\newtheorem{corollary}[thm]{Corollary}
\newtheorem{lemma}[thm]{Lemma}

\theoremstyle{definition}
\newtheorem{defn}[thm]{Definition}
\theoremstyle{remark}
\newtheorem{rem}[thm]{Remark}
\numberwithin{equation}{section}

\DeclareMathOperator*{\osc}{osc}

\def\Xint#1{\mathchoice
      {\XXint\displaystyle\textstyle{#1}}%
      {\XXint\textstyle\scriptstyle{#1}}%
      {\XXint\scriptstyle\scriptscriptstyle{#1}}%
      {\XXint\scriptscriptstyle\scriptscriptstyle{#1}} %
\!\int}
   \def\XXint#1#2#3{{\setbox0=\hbox{$#1{#2#3}{\int}$}
        \vcenter{\hbox{$#2#3$}}\kern-.5\wd0}}
   
   \def\dashint{\Xint-}

\def\dxy{\,{\rm d}x{\rm d}y}

\def\eps{\varepsilon}

\allowdisplaybreaks
\makeatletter
\DeclareRobustCommand*{\bfseries}{%
  \not@math@alphabet\bfseries\mathbf
  \fontseries\bfdefault\selectfont
  \boldmath
}
\makeatother

\newlength{\defbaselineskip}
\newcommand{\setlinespacing}[1]
           {\setlength{\baselineskip}{#1 \defbaselineskip}}

\title
[Local behavior of fractional $p$-minimizers]
{Local behavior of fractional $p$-minimizers}

\author[A. Di Castro]{Agnese Di Castro}
\email[Agnese Di Castro]{\href{mailto:dicastro@mail.dm.unipi.it}{dicastro@mail.dm.unipi.it}} 
\author[T. Kuusi]{Tuomo Kuusi}
\email[Tuomo Kuusi]{\href{mailto:tuomo.kuusi@aalto.fi}{tuomo.kuusi@aalto.fi}}
\author[G. Palatucci]{Giampiero Palatucci}
\email[Giampiero Palatucci]{\href{mailto:giampiero.palatucci@unipr.it}{giampiero.palatucci@unipr.it}}

\address[A. Di Castro]
{Dipartimento di Matematica e Informatica, Universit\`a degli Studi di Parma
\\ Campus - Parco Area delle Scienze~53/A
\\ 43124 Parma, Italy; \newline Dipartimento di Matematica, Universit\`a degli Studi di Pisa
\\ Largo B. Pontecorvo 5
\\ 56127 Pisa, Italy}

\address[G. Palatucci]
{Dipartimento di Matematica e Informatica, Universit\`a degli Studi di Parma
\\ Campus - Parco Area delle Scienze~53/A
\\ 43124 Parma, Italy; \newline
SISSA 
 \\ Via Bonomea 256 \\ 34136 Trieste, Italy
}

\address[T. Kuusi]
{Department of Mathematics and Systems Analysis, Aalto University
\\ P.O. Box 1100
\\ 00076 Aalto, Finland
}

\begin{document}


\keywords{Quasilinear nonlocal operators, fractional Sobolev spaces, H\"older regularity, Caccioppoli estimates, singular perturbations\vspace{0.5mm}}


\begin{abstract}
\small
We extend the De Giorgi-Nash-Moser theory to nonlocal, possibly degenerate 
integro-differential operators.
\end{abstract}

\maketitle

\begin{center}
 \rule{11.3cm}{0.6pt}\\[-0.0cm] 
{\sc {\small To appear in}\, {\it 
Ann. Inst. H. Poincare Anal. Non Lineaire}. 
}
\\[-0.19cm] \rule{11.3cm}{0.6pt}
\end{center}
\vspace{4mm}

\setcounter{tocdepth}{2}
 \setlinespacing{1.03}

\section{Introduction}

The aim of this paper is to develop localization techniques in order to establish regularity results for nonlocal integro-differential operators and minimizers  
of fractional order $s\in (0,1)$ and summability $p>1$. 
Let ${\Omega}$ be a bounded domain and let~$g$ be a function in the fractional Sobolev space $W^{s,p}({\mathds R}^n)$. We shall prove general local regularity estimates for the minimizers~$u$, where $u$ is minimizing the functional 
\begin{equation}\label{def_f1}
{\displaystyle}
{\mathcal F} (v):= \int_{{\mathds R}^n} \int_{{\mathds R}^n} K(x,y)|v(x)-v(y)|^p \dxy, \quad 
\end{equation}
over the class of functions $ \{v \in W^{s,p}({\mathds R}^n) \, : \, v = g \, \mbox{ a.e. in }  {\mathds R}^n \setminus \Omega \}$.
Here $K$ is a suitable symmetric kernel of order $(s,p)$ with just measurable coefficients, see~\eqref{hp_K}. It is standard to show, which is in fact our Theorem~\ref{lem_sss} below, that minimizers can be equivalently characterized by the weak solutions to the following class of integro-differential problems
\begin{equation}\label{p1}
\begin{cases}
\mathcal{L} u=0 & \text{in }\Omega,\\
u=g & \text{in }\mathds{R}^n\setminus\Omega,
\end{cases}
\end{equation}
where the operator $\mathcal{L}$ is defined formally by
\begin{equation}\label{taneli}
\mathcal{L}u(x)=P.~\!V.\int_{\mathds{R}^n} K(x,y)|u(x)-u(y)|^{p-2}(u(x)-u(y))\,{\rm d}y, \quad x\in \mathds{R}^n;
\end{equation}
the symbol $P.~\!V.$ means ``in the principal value sense''.
We immediately refer to Section~\ref{sec_pre} for the precise assumptions on the involved quantities.

To simplify, one can keep in mind the model case when the kernel $K(x,y)$  coincides with $|x-y|^{-(n+sp)}$, though in such a case the difficulties arising from having merely measurable coefficients disappear; that is, the function $u$ reduces to the solution of the following problem
\begin{equation}\label{def_p1}
\begin{cases}
{\displaystyle}
(-\Delta)^{s}_p \,u = 0 & \text{in} \ \Omega, \\[1ex]
u = g & \text{in} \ {\mathds R}^n\setminus \Omega,
\end{cases}
\end{equation}
where the symbol~${\displaystyle} (-\Delta)^{s}_p$ denotes the standard fractional $p$-Laplacian operator.

Recently, a great attention has been focused on the study of problems involving fractional Sobolev spaces and corresponding nonlocal equations, both from a pure mathematical point of view and for concrete applications, since they naturally arise in many different contexts. For an elementary introduction to this topic and 
for a quite extensive list of related references we refer to~\cite{DPV12}.

However, for what concerns regularity and related results for this kind of operators when $p\neq 2$,  
the theory seems to be rather incomplete. Nonetheless, some partial results are known. Firstly, we would like to cite the higher regularity contributions for viscosity solutions in the case when {\it $s$ is close to $1$} proven in the recent interesting paper~\cite{BCF12}; see, also,~\cite{IN10}. Secondly, the analysis in the papers~\cite{CLM12} and \cite{LL13} considers the special case when {\it $p$ is suitably large} - thus falling in the  
 Morrey embedding case when concerning regularity. See also~\cite{FP13} for some basic results for fractional $p$-eigenvalues.

On the contrary, when $p=2$ and $K(x,y) = |x-y|^{-n-2s}$, that is the case of the well-known fractional Laplacian operator~$(-\Delta)^s$, the situation simplifies notably. Although having been a classical topic in  Functional and Harmonic Analysis as well as in Partial Differential Equations for a long time,
in the last years the growing interest for such operator has become really significant and many important results for the minimizer of~\eqref{def_f1} have been achieved.
 For what concerns the main topic in the present paper, i.e., the local behavior of the fractional minimizers, it is worth mentioning the very relevant contributions for the case $p=2$ by Kassmann (\cite{Kas09,Kas12}); see also \cite{Sil06,SV13}. In particular, among other results,  Kassmann proves H\"older regularity and a Harnack inequality ``revisited'' in the right form taking into account the nonlocality of the fractional Laplacian operator; we refer also to~\cite{Kas07} to discover how the classic Harnack inequality fails in the fractional framework.
 
In the present paper, we will deal with a larger class of operators with a symmetric kernel $K$ having only measurable coefficients, and, above all, satisfying fractional differentiability for {\it any} $s\in (0,1)$ and $p$-summability for  {\it any} $p>1$. For this, we will have to handle not only the usual nonlocal character of such fractional operators, but also the difficulties given by the corresponding nonlinear behavior. 
As a consequence, we can make use neither of the powerful framework provided by the Caffarelli-Silvestre $s$-harmonic extension (\cite{CS07}) nor of various  tools as, e.~\!g., 
the sharp 3-commutators estimates introduced in~\cite{DLR11} to deduce the regularity of weak fractional harmonic maps,
the strong barriers and density estimates in~\cite{PSV13,SV13a,SV13b}, 
the commutator and energy estimates in~\cite{PP13,PS13}, and so on. Indeed, the aforementioned tools seem not to be trivially adaptable to a nonlinear framework; also, increasing difficulties are due to the non-Hilbertian structure of the involved fractional Sobolev spaces~$W^{s,p}$ when $p$ is different than 2.

We will have to work carefully in order to obtain the needed local estimates. 
For this, we want to underline that a specific quantity will be fundamental throughout the whole paper. Namely, we introduce the {\it nonlocal tail} of a function $v \in W^{s,p}({\mathds R}^n)$ in the ball $B_{R}(x_0)\subset {\mathds R}^n$ given by
\begin{equation}\label{def_t1}
{\rm Tail}(v; x_0, R):=\left[R^{sp}\int_{\mathds{R}^n\setminus B_R(x_0)}|v(x)|^{p-1}|x-x_0|^{-(n+sp)}\,{\rm d}x\right]^{\frac{1}{p-1} }.
\end{equation} 
Note that the above number is finite by H\"older's inequality whenever $v \in L^{q}({\mathds R}^n)$, $q\geq p-1$, and $R>0$.
As expected, the way how the nonlocal tail will be  managed is a key-point in the present extended local theory. We believe that this is a general fact that will have to be taken into account in other results and extensions in the nonlinear fractional framework. 

We are now ready to introduce our main results. The first one describes the local boundedness. 
\begin{thm}[\bf Local boundedness]\label{sup}
Let $p \in (1,\infty)$, let $u\in W^{s,p}(\mathds{R}^n)$ be a weak subsolution to problem~\eqref{p1} and let $B_r\equiv B_r(x_0)  \subset \Omega$. Then the following estimate holds true
\begin{eqnarray}\label{sup_estimate}
\sup_{B_{r/2}(x_0)}u \, \leq \, \delta\, {\rm Tail}(u_+;x_0,r/2)+c \delta^{-\frac{(p-1)n}{s p^2}} \left(\dashint_{B_r(x_0)} u_+^p\,{\rm d}x\right)^{\frac 1p},
\end{eqnarray}
where ${\rm Tail}(u_+;x_0,r/2)$ is defined in~\eqref{def_t1}, $u_+=\max\{u,0\}$ is the positive part of the function $u$, $\delta \in (0,1]$, and the constant $c$ depends only on $n,p,s$ and on the structural constants $\lambda,\Lambda$ defined in~\eqref{hp_K}.
\end{thm} 
The parameter $\delta$ allows interpolation between the local and nonlocal terms. Armed with the Logarithmic Lemma and the Caccioppoli estimate with tail introduced below, together with the deduced local boundedness, we can prove our main result, that is, the H\"older continuity theorem.
\begin{thm}[\bf H\"older continuity]\label{holder1}
Let $p\in(1,\infty)$ and let $u\in W^{s,p}(\mathds{R}^n )$ be a solution to  problem~\eqref{p1}. Then $u$ is locally H\"older continuous in $\Omega$. In particular, 
there are positive constants  $\alpha$, $\alpha < sp/(p-1)$, and $c$, both depending only on $n,p,s,\lambda,\Lambda$, such that if $B_{2r}(x_0) \subset \Omega$, then
\begin{equation*}
\osc_{B_\varrho(x_0)} u \leq c \left(\frac{\varrho}{r}\right)^\alpha \left[ {\rm Tail}(u;x_0,r)+ \left(\dashint_{B_{2r}(x_0)} |u|^p\,{\rm d}x\right)^{\frac 1p} \right]
\end{equation*}
holds whenever $\varrho \in (0,r]$.
\end{thm}

The theorem above provides an extension of classical analogous results by De Giorgi-Nash-Moser~(\cite{Deg57,Nas58,Mos61})  to the nonlocal, nonlinear framework. It also extends  the recent aforementioned result by Kassmann (\cite{Kas09}) to the case $p\neq 2$.
Moreover, it is worth noticing that in the linear case studied in~\cite{Kas09} a further boundedness assumption is required, which is now for free thanks to Theorem~\ref{sup}.

In the proof of the H\"older continuity the following {\it logarithmic estimate} plays the key role. We state it in the introduction as we think that it might be extremely 
 useful also in other contexts.  
\begin{lemma}[\bf Logarithmic Lemma]\label{log_lemma}
Let $p \in (1,\infty)$. Let $u\in W^{s,p}(\mathds{R}^n)$ be a weak supersolution to  problem~\eqref{p1} such that $u\geq 0$ in $B_R \equiv B_R(x_0) \subset \Omega$. Then the following estimate holds for any $B_{r}\equiv B_{r}(x_0) \subset B_{R/2}(x_0)$ and any $d>0$, 
\begin{eqnarray}\label{log_estimate}
&& \int_{B_{r}} \int_{B_{r}} K(x,y) \left|\log\left(\frac{d + u(x)}{d + u(y)}\right)\right|^p \,{\rm d}x {\rm d}y \nonumber \\[1ex]
&& \qquad \qquad \qquad \quad \quad \leq c\,r^{n-sp}\left\{  d^{1-p}\,\left(\frac{r}{R}\right)^{sp}\,\left[{\rm Tail}(u_-;x_0,R)\right]^{p-1}+1\right\},
\end{eqnarray}
where ${\rm Tail}(u_-;x_0,R)$ is defined in~\eqref{def_t1}, $u_{-}=\max\{-u,0\}$ is the negative part of the function $u$, and $c$ depends only on $n$, $p$, $s$, $\lambda$ and $\Lambda$. 
\end{lemma}

Then, we will show that the fractional $p$-minimizers, equivalently the weak solutions to the  Euler-Lagrange equation associated to~\eqref{def_f1}, satisfy the following nonlocal Caccioppoli-type inequalities.
\begin{thm}[{\bf Caccioppoli estimates with tail}]\label{cacio} Let $p \in (1,\infty)$ and
let $u\in W^{s,p}(\mathds{R}^n)$ be a weak solution to problem \eqref{p1}.
 Then, for any $B_r\equiv B_r(x_0)\subset \Omega$ and any nonnegative $\phi\in C^\infty_0(B_r)$,  the following estimate holds true
\begin{eqnarray}\label{cacio1}
\nonumber && \int_{B_r}\int_{B_r}K(x,y) |w_{\pm}(x)\phi(x)-w_{\pm}(y)\phi(y)|^p \,{\rm d}x{\rm d}y\\
 &&\qquad \qquad  \leq c\int_{B_r}\int_{B_r} K(x,y)
 (\max\{w_{\pm}(x),w_\pm(y)\})^p |\phi(x)-\phi(y)|^p\,{\rm d}x{\rm d}y\\
&&\qquad \qquad \quad+\,c \,\int_{B_r}w_{\pm}(x)\phi^p(x)\,{\rm d}x \left(\sup_{y\,\in\, {\rm supp}\,\phi}\int_{\mathds{R}^n\setminus B_r} K(x,y)w_{\pm}^{p-1}(x)\,{\rm d}x \right)\!, 
\nonumber
\end{eqnarray}
where $w_{\pm}:=(u-k)_{\pm}$ and $c$ depends only on~$p$.
\end{thm}
\begin{rem}\label{rem_sub} The estimate in~\eqref{cacio1} continues to hold for $w_+$  when $u$ is merely a weak subsolution to~\eqref{pb1} and for $w_-$ when $u$ is a weak supersolution to~\eqref{pb1}. 
\end{rem}
Notice that, as expected, in the nonlocal framework one has to take into account a suitable tail; see, in particular, the estimate in~\eqref{taaa} below to see how the second term in the right hand-side of~\eqref{cacio1} is controlled by a tail as given in definition~\eqref{def_t1}.
Also, it is worth mentioning that   other fractional Caccioppoli-type inequalities 
have been recently used in different contexts (see, for instance,~\cite{Min07,Min11,FP13}), 
 although none of them takes into account the tails.\footnote{
 We recently discovered that Kassmann proved similar Caccioppoli estimates with tail terms in the linear case, when $p=2$; see~\cite{Kas07b}.
 }

\vspace{2mm}

Let us finally comment some recent results in the literature. In~\cite{DKP14} we prove Harnack-type inequalities with tail for weak supersolutions and solutions to~\eqref{p1}. These can be applied to obtain H\"older continuity of the solutions. However, the proof in~\cite{DKP14} are heavily based on the tools developed in the present paper.
Moreover, the regularity theory for the inhomogeneous counterpart $\mathcal{L}u=f$ have been settled in~\cite{KMS15} in a general setting, including also the case when the source term~$f$ is merely a measure. In turn, these results are partly based on the quantitative estimates established here.  The principal value definition~\eqref{taneli} has been used in \cite{Lin14} to obtain regularity results in the context of viscosity solutions. Also, for general existence results and other regularity issues, we refer  to the very recent contributions in~\cite{ILPS14}, and in~\cite{BP15}, where the related fractional $p$-eigenvalue problem has been considered.

\vspace{2mm}
The paper is organized as follows. In Section~\ref{sec_pre} below, we fix the notation by also providing some preliminary results. Section~\ref{sec_log} is devoted to the proof of the Log~Lemma \ref{log_lemma} and the Caccioppoli estimates with tail in Theorem~\ref{cacio}. In Section~\ref{sec_sup}, we establish the local boundedness given by Theorem~\ref{sup}.
In Section~\ref{sec_holder}, we shall finally prove the H\"older continuity given by Theorem~\ref{holder1}.

\vspace{3mm}
\section{Preliminaries}\label{sec_pre}

In this section, we state the general assumptions of the problem  we deal with in the present paper, we fix notation, and we provide some definitions and some basic preliminary results that we will use in the following pages.

The kernel $K:\mathds{R}^n\times \mathds{R}^n \to [0,\infty)$ is a symmetric measurable function such that
\begin{equation}\label{hp_K}
\lambda\leq K(x,y)|x-y|^{n+sp}\leq \Lambda \,\,\,\text{for almost every} \ x,\,y\in \mathds{R}^n,
\end{equation}
for some $s\in(0,1)$, $p>1$, $\Lambda\geq\lambda\geq 1$. Notice that such assumption on $K$ can be weakened as follows
\begin{equation*} 
\lambda \leq K(x,y)|x-y|^{n+sp} \leq \Lambda \ \text{for almost every} \ x, y \in {\mathds R}^n \ \text{s.~\!t.} \ |x-y| \leq 1,
\end{equation*}
\begin{equation*} 
0\leq K(x,y)|x-y|^{n+\eta} \leq M \ \text{for almost every} \ x, y \in {\mathds R}^n \ \text{s.~\!t.} \ |x-y| > 1,
\end{equation*}
for some $s, \lambda, \Lambda$ as above, $\eta>0$ and $M\geq 1$; see, e.~\!g., \cite{Kas09,Kas12}. For the sake of simplicity, we will keep the assumption in~\eqref{hp_K}, since such a choice will imply no relevant differences in all the proofs in the rest of the paper.

For any $p\in [1,\infty)$ and $s\in(0,1)$ we denote by $W^{s,p}({\mathds R}^n)$ the fractional Sobolev space, that is
$$
W^{s,p}({\mathds R}^n):= \left\{v\in L^p({\mathds R}^n) : \frac{|v(x)-v(y)|}{|x-y|^{\frac np+s}}\in L^p({\mathds R}^n\times {\mathds R}^n)\right\};
$$
i.~\!e., an intermediary Banach space between $L^p({\mathds R}^n)$ and $W^{1,p}({\mathds R}^n)$, endowed with the natural norm
$$
\|v\|_{W^{s,p}({\mathds R}^n)}:= \left(\int_{{\mathds R}^n}|v|^p\,{\rm d}x\right)^{\frac 1p}+\left(\int_{{\mathds R}^n}\int_{{\mathds R}^n} \frac{|v(x)-v(y)|^p}{|x-y|^{n+sp}}\,{\rm d}x{\rm d}y\right)^{\frac 1p}.
$$
In a similar way, it is possible to define the fractional Sobolev space $W^{s,p}({\Omega})$ in a domain $\Omega\subseteq\mathds{R}^n$. 
Furthermore, by saying that $v$ belongs to $W_0^{s,p}({\Omega})$ we mean that $v \in W^{s,p}({\mathds R}^n)$ and $v=0$ almost everywhere in ${\mathds R}^n \setminus \Omega$. 

As mentioned in the introduction, we  define the {\it nonlocal tail of a function $v$ in the ball $B_{R}(x_0)$}, a quantity which will play an important role in the rest of the paper. For any $v\in W^{s,p}(\mathds{R}^n)$ and $B_R(x_0)\subset \mathds{R}^n$, we write
\begin{equation}\label{Tail}
{\rm Tail}(v; x_0, R):=\left[R^{sp}\int_{\mathds{R}^n\setminus B_R(x_0)}|v(x)|^{p-1}|x-x_0|^{-n-sp}\,{\rm d}x\right]^{\frac{1}{p-1} },
\end{equation} 
which is a finite number by H\"older's inequality since $v \in L^{p}({\mathds R}^n)$ and $R>0$.
 

Let $\Omega$ be a bounded open set in $\mathds{R}^n$ and $g\in W^{s,p}(\mathds{R}^n)$, we will be interested in weak solutions to the following integro-differential problems  
\begin{equation}\label{pb1}
\begin{cases}
\mathcal{L} u=0 & \text{in }\Omega,\\
u=g & \text{in }\mathds{R}^n\setminus\Omega,
\end{cases}
\end{equation}
where the operator~$\mathcal{L}$ is formally defined in~\eqref{taneli}.
Notice that the boundary condition is given in the whole complement of $\Omega$, as usual when dealing with such nonlocal operators. A model example we have in mind is 
the fractional $p$-Laplacian, that is
$$
-(-\Delta)^s_p\, u(x)
\, =\, 
c(n,s,p)\ P.~\!V.\!\int_{\mathds{R}^n} \frac{|u(x)-u(y)|^{p-2} (u(x)-u(y))}{|x-y|^{n +sp}}\dxy,
$$
with $s\in (0,1)$ and $p>1$.

Now, let us consider in $W^{s,p}(\mathds{R}^n)$, the following functional
\begin{equation}\label{def_F}
\mathcal{F}(u)=\int_{\mathds{R}^n}\int_{\mathds{R}^n} K(x,y)|u(x)-u(y)|^p\,{\rm d}x{\rm d}y.
\end{equation}
In view of the assumptions~\eqref{hp_K} on $K$, one can use the standard Direct Method to prove that there exists a unique $p$-minimizer of $\mathcal{F}$ over all $u\in W^{s,p}(\mathds{R}^n)$ such that $u(x)=g(x)$ for $x\in \mathds{R}^n\setminus \Omega$.
Moreover, a $p$-minimizer $u$ is a weak solution solution to problem \eqref{pb1} and vice versa  
 (see Theorem~\ref{lem_sss} below). 
 
 To specify relevant spaces, for given $g \in W^{s,p}(\mathds{R}^n)$, we define the convex sets of $W^{s,p}(\mathds{R}^n)$ as
 \[
 \mathcal{K}_g^\pm (\Omega) := \left\{ v \in W^{s,p}(\mathds{R}^n) \; : \; (g-v)_\pm \in  W_0^{s,p}(\Omega) Ê\right\}
 \]
 and
 \[
  \mathcal{K}_g (\Omega) :=   \mathcal{K}_g^+ (\Omega) \cap  \mathcal{K}_g^- (\Omega) =  \left\{ v \in W^{s,p}(\mathds{R}^n) \; : \; v-g  \in  W_0^{s,p}(\Omega) Ê\right\}.
 \]
We recall 
that the functions in the space $W^{s,p}_0({\Omega})$ are defined in the whole space, since they are considered to be extended to zero outside~${\Omega}$.

We conclude this section by recalling the definition of weak sub- and supersolutions as well as weak solutions to problem~\eqref{pb1}.
\begin{defn}\label{def_weak} Let $g \in W^{s,p}(\mathds{R}^n)$. A function $u \in  \mathcal{K}_g^-$ $(\mathcal{K}_g^+)$ is a weak {\it subsolution} ({\it supersolution}) to problem \eqref{pb1} if  
\begin{equation}\label{weak_sub}
\int_{\mathds{R}^n}\int_{\mathds{R}^n} K(x,y)|u(x)-u(y)|^{p-2}(u(x)-u(y))(\eta(x)-\eta(y))\,{\rm d}x{\rm d}y \, \leq \, (\geq) \; 0
\end{equation}
for every nonnegative $\eta\in W_0^{s,p}(\Omega) $. 

A function $u$ is a weak {\it solution} to problem \eqref{pb1} if it is both weak sub- and supersolution. In particular, $u$ belongs to $\mathcal{K}_g (\Omega)$ and satisfies 
\begin{equation}\label{weak_sol}
\int_{\mathds{R}^n}\int_{\mathds{R}^n} K(x,y)|u(x)-u(y)|^{p-2}(u(x)-u(y))(\eta(x)-\eta(y))\,{\rm d}x{\rm d}y =  0
\end{equation}
for every $\eta\in W_0^{s,p}(\Omega) $.
 \end{defn}

Similarly, we recall the definition of sub- and superminimizers of \eqref{def_F}. We have
\begin{defn}\label{def_min}
Let $g \in W^{s,p}(\mathds{R}^n)$. A function $u\in \mathcal{K}_g^-$ is a {\it subminimizer} of the functional \eqref{def_F} over $\mathcal{K}_g^-$ if
$
{\displaystyle} \mathcal{F}(u)\leq \mathcal{F}(u+\eta)
$ 
for every nonpositive $\eta\in W_0^{s,p}(\Omega)$. Similarly, a function~$u\in \mathcal{K}_g^+$  is a {\it superminimizer} of the functional \eqref{def_F} over $\mathcal{K}_g^+$ if
$
{\displaystyle} \mathcal{F}(u)\leq \mathcal{F}(u+\eta)
$
for every nonnegative $\eta\in W_0^{s,p}(\Omega)$.

Finally, $u\in \mathcal{K}_g$ is a {\it minimizer} of the functional \eqref{def_F} over $\mathcal{K}_g$ if
$
{\displaystyle} \mathcal{F}(u)\leq \mathcal{F}(u+\eta)
$
for every $\eta\in W_0^{s,p}(\Omega)$.
 \end{defn}

\subsection{Notation}
Before starting with the proofs, it is convenient to fix some notation which will be used throughout the rest of the paper. Firstly, notice that we will follow the usual convention of denoting by $c$ a general positive constant which will not necessarily be the same at different occurrences and which can also change from line to line. For the sake of readability, dependencies of the constants will  be often omitted within the chains of estimates, therefore stated after the estimate. 
Relevant dependences on parameters will be emphasized by using parentheses; special constants will be denoted by $c_0$, $c_1$,....

As customary, we denote by
$$
B_R(x_0)=B(x_0; R):=\Big\{x\in\mathds{R}^n : |x-x_0|<R\Big\}
$$
the open ball centered in $x_0\in\mathds{R}^n$ with radius $R>0$. When not important and clear from the context, we shall use the shorter notation $B_R:=B(x_0;R)$. We denote by $\beta B_R$ the concentric ball scaled by a factor $\beta>0$, that is $\beta B_R:=B(x_0;\beta R)$. Moreover, if $f\in L^1 (S)$ and the $n$-dimensional Lebesgue measure $|S|$ of the set $S\subseteq \mathds{R}^n$ is finite and strictly positive, we write
\begin{equation}\label{mean}
(f)_S:=\dashint_{S} f(x)\,{\rm d}x=\frac{1}{|S|}\int_S f(x)\,{\rm d}x.
\end{equation}

Let $k\in \mathds{R}$, we denote by
\begin{equation}\label{w+}
w_+(x):=(u(x)-k)_+=\max\{ u(x)-k,0\},  
\end{equation}
and
\begin{equation}\label{w-}
w_-(x):=(u(x)-k)_-=(k-u(x))_+.  
\end{equation}
Clearly $w_+(x)\neq 0$ in the set $\big\{x\in S:u(x)>k\big\}$, and $w_-(x)\neq 0$ in the set \newline$\big\{x\in S:u(x)<k\big\}$. 
\vspace{1.5mm}


\vspace{3mm}

\subsection{Existence and uniqueness of the minimizers}
The proof of the existence and uniqueness for fractional minimizers is simple and it is recorded into the following. 

\begin{thm}\label{lem_sss}
Let $s\in (0,1)$ and $p \in [1,\infty)$, and let $g \in W^{s,p}({\mathds R}^n)$. Then there exists a minimizer $u$ of~\eqref{def_F} over $\mathcal{K}_g$. Moreover, if $p>1$, then the solution is unique. Moreover, a function $u\in\mathcal{K}_g$ is a minimizer of~\eqref{def_F} over $\mathcal{K}_g$ if and only if it is a weak solution to problem~\eqref{pb1}. 
\end{thm}
\begin{proof}
The proof plainly follows by the Direct Method of Calculus of Variations. One can take any minimizing sequence~$u_j\in \mathcal{K}_g$. Due to the assumptions on the kernel $K$, one can control the fractional seminorm of~$u_j$, so that, one can find by pre-compactness in~$L^{p}$ (see, for instance,~\cite[Theorem 6.7]{DPV12})
a subsequence~$u_{j_k}$ converging pointwise a.~\!e. to a function $u\in \mathcal{K}_g$. By Fatou's Lemma we deduce that~$u$ is actually a minimizer of~\eqref{def_F} over $\mathcal{K}_g$.  The uniqueness in the case $p>1$ follows from the strict convexity of the functional.

Furthermore, the fact that $u$ solves the corresponding Euler-Lagrange equation follows by perturbing $u\in \mathcal{K}_g$ with a test function in a standard way. 
Indeed, supposing that $u\in \mathcal{K}_g$ is a minimizer of~\eqref{def_F} over $\mathcal{K}_g$, take any $\phi \in W_0^{s,p}(\Omega)$ and calculate formally
\begin{align} 
\notag 
  \left. \frac{d}{dt}  \mathcal{F}(u+t\phi) \right|_{t=0}  \ = \ 
\left.   \int_{\mathds{R}^n}\int_{\mathds{R}^n} K(x,y)\frac{\rm d}{{\rm d}t} |u(x)-u(y) + t (\phi(x)-\phi(y))|^p  \,{\rm d}x{\rm d}y \right|_{t=0} 
\\ \notag  \qquad =  \ p \int_{\mathds{R}^n}\int_{\mathds{R}^n} K(x,y) |u(x)-u(y)|^{p-2}(u(x)-u(y)) (\phi(x)-\phi(y))  \,{\rm d}x{\rm d}y \,.
\end{align}
Since $u$ is a minimizer, the term on the left is zero and hence $u\in \mathcal{K}_g$ is a weak solution to problem~\eqref{pb1}. 
For the converse, let $u\in \mathcal{K}_g$ be a weak solution to problem~\eqref{pb1} and take $\phi = u-v\in W_0^{s,p}(\Omega)$, where $v \in  \mathcal{K}_g$. Then, by Young's inequality,
\begin{eqnarray*} 
\notag 0 & = & \int_{\mathds{R}^n}\int_{\mathds{R}^n} K(x,y) |u(x)-u(y)|^{p-2}(u(x)-u(y)) (\phi(x)-\phi(y))  \,{\rm d}x{\rm d}y 
\\ \notag & = & \int_{\mathds{R}^n}\int_{\mathds{R}^n} K(x,y) |u(x)-u(y)|^{p}\,{\rm d}x{\rm d}y 
\\ \notag & & - \int_{\mathds{R}^n}\int_{\mathds{R}^n} K(x,y) |u(x)-u(y)|^{p-2}(u(x)-u(y))(v(x)-v(y))\,{\rm d}x{\rm d}y 
\\ \notag & \geq & \frac1p \int_{\mathds{R}^n}\int_{\mathds{R}^n} K(x,y) |u(x)-u(y)|^{p}\,{\rm d}x{\rm d}y 
\\ \notag & &  -  \frac1p  \int_{\mathds{R}^n}\int_{\mathds{R}^n} K(x,y) |v(x)-v(y)|^{p}\,{\rm d}x{\rm d}y \,,
\end{eqnarray*}
and hence $u$ is a minimizer of~\eqref{def_F} over $\mathcal{K}_g$.
\end{proof}

\vspace{3mm}

\section{Fundamental estimates}\label{sec_log}
In this section, we establish some relevant estimates that we will use in the following. We believe that these results could have their own interest in the analysis of equations involving the (nonlinear) fractional Laplacian and related nonlocal operators. The first of them states a natural extension of the well-known Caccioppoli inequality to the nonlocal framework, by showing that in such a case one can  take into account a suitable tail, in order to detect deeper informations.  
\begin{proof}[Proof of {Theorem~\ref{cacio}}] 
For the sake of generality, we would point out that the present proof  is also valid when $p=1$.

Let $u$ be a weak solution as in the statement. Testing \eqref{weak_sub} with $\eta:=w_+\,\phi^p$, where $\phi$ is any nonnegative function in $C^\infty_0(B_r(x_0))$, we get 
\begin{eqnarray}\label{eq1}
{\displaystyle}
0 & \geq & \int_{B_r}\int_{B_r} K(x,y)|u(x)-u(y)|^{p-2} \\
&& \qquad \quad \ \times (u(x)-u(y))(w_+(x)\phi^p(x)-w_+(y)\phi^p(y))\,{\rm d}x{\rm d}y \nonumber
\\[1ex]
&&  + 2 \int_{\mathds{R}^n\setminus B_r}\int_{B_r} K(x,y)|u(x)-u(y)|^{p-2}\nonumber\\
&& \qquad \qquad \qquad \times (u(x)-u(y))w_+(x)\phi^p(x)\,{\rm d}x{\rm d}y
\nonumber 
\end{eqnarray}
Note that $\eta$ is an admissible test function since truncations of functions in $W^{s,p}({\mathds R}^n)$ still belong to  $W^{s,p}({\mathds R}^n)$.

Let us consider the integrands of the two terms above separately. In the first term, we may assume without loss of generality that $u(x) \geq u(y)$; otherwise just exchange the roles of $x$ and $y$ below. We have
\begin{eqnarray*} 
\notag && |u(x)-u(y)|^{p-2} (u(x)-u(y))(w_+(x)\phi^p(x)-w_+(y)\phi^p(y))
\\[1ex]  \notag & &   \quad = (u(x)-u(y))^{p-1} ((u(x)-k)_+ \phi^p(x) - (u(y)-k)_+\phi^p(y))
\\[1ex]  \notag & & \quad = \left\{ 
\begin{array}{rl}
  (w_+(x)-w_+(y))^{p-1} (w_+(x) \phi^p(x) - w_+(y)\phi^p(y))\,,  &  u(x),u(y) > k   \\
(u(x)-u(y))^{p-1} w_+(x) \phi^p(x)\,,   &     u(x) >  k\,, \; u(y) \leq k  \\
  0, &   \mbox{otherwise}   
\end{array}
\right.
\\[1ex]  \notag & & \quad \geq 
  (w_+(x)-w_+(y))^{p-1} (w_+(x) \phi^p(x) - w_+(y)\phi^p(y)).
\end{eqnarray*}
For the second term in the right hand-side of the inequality in~\eqref{eq1} we instead have
\begin{eqnarray*} 
\notag |u(x)-u(y)|^{p-2}(u(x)-u(y))w_+(x)  & \geq & - (u(y)-u(x))_+^{p-1} (u(x)-k)_+
\\  \notag  \qquad & \geq & - (u(y)-k)_+^{p-1} (u(x)-k)_+
\\  \notag  \qquad & = & -w_+(y)^{p-1} w_+(x),
\end{eqnarray*}
and estimating further we obtain
\begin{eqnarray*}
&& \int_{\mathds{R}^n\setminus B_r}\int_{B_r} K(x,y)|u(x)-u(y)|^{p-2}(u(x)-u(y))w_+(x)\phi^p(x)\,{\rm d}x{\rm d}y\\
&& \qquad \qquad\quad\quad   \geq -\int_{\mathds{R}^n\setminus B_r}\int_{B_r} K(x,y)w_+^{p-1}(y)w_+(x)\phi^p(x)\,{\rm d}x{\rm d}y\\
&&\quad\qquad\quad\quad\quad   \geq -\left(\sup_{x\in{\rm supp}\,\phi}\int_{\mathds{R}^n\setminus B_r} K(x,y)w_+^{p-1}(y)\,{\rm d}y\right) \int_{B_r} w_+(x)\phi^p(x)\,{\rm d}x.
\end{eqnarray*} 
We thus deduce from~\eqref{eq1} that
\begin{eqnarray}\label{eq111111}
{\displaystyle}
0 & \geq &\ \int_{B_r}\int_{B_r} K(x,y)|w_+(x)-w_+(y)|^{p-2} \\
&& \qquad \qquad  \times (w_+(x)-w_+(y))(w_+(x)\phi^p(x)-w_+(y)\phi^p(y))\,{\rm d}x{\rm d}y \nonumber
\\[1ex]
&& \notag \quad  -2 \left(\sup_{x\in{\rm supp}\,\phi}\int_{\mathds{R}^n\setminus B_r} K(x,y)w_+^{p-1}(y)\,{\rm d}y\right) \int_{B_r} w_+(x)\phi^p(x)\,{\rm d}x. 
\end{eqnarray}
 
Let us then consider the first term in the  inequality above. If $w_+(x) \geq w_+(y)$ and $\phi(x) \leq \phi(y)$ in the integrand, we appeal to Lemma~\ref{lemma:trivial conv} below and get 
\begin{equation}\label{eq_11gat}
\phi^p(x)\geq (1- c_p \,\varepsilon )\, \phi^p(y)-(1+c_p \eps) \,\varepsilon^{1-p}|\phi(x)-\phi(y)|^p
\end{equation}
for any $\varepsilon\in (0,1]$ with the constant $c_p \equiv (p-1)\Gamma(\max\{1,p-2\})$.
Thus, by choosing 
$$
\varepsilon := \frac{1}{\max\{1,2 c_p\}} \,\frac{w_+(x)-w_+(y)}{w_+(x)}\in(0,1]
$$ 
we get
 \begin{eqnarray*} 
\notag
(w_+(x)-w_+(y))^{p-1} w_+(x)\phi^p(x) & \geq & (w_+(x)-w_+(y))^{p-1} w_+(x) (\max\{\phi(x),\phi(y)\})^p
\\ \notag & & - \frac12 \,(w_+(x)-w_+(y))^p (\max\{\phi(x),\phi(y)\})^p
\\ \notag & & - c(\max\{w_+(x),w_+(y)\})^p |\phi(x)-\phi(y)|^p
\end{eqnarray*}
with $c \equiv c(p)$. Recall that in the estimate above we assumed that $\phi(x) \leq \phi(y)$, $\max\{\phi(x),\phi(y)\} = \phi(y)$. However, when $0 = w_+(x) \geq w_+(y) \geq 0$ or $w_+(x) \geq w_+(y)$ and $\phi(x)\geq \phi(y)$, the estimate in the  display above is trivial and hence we conclude that it holds also in these cases. It follows that
\begin{eqnarray*} 
\notag && \hspace{-0.5cm}(w_+(x)-w_+(y))^{p-1} (w_+(x)\phi^p(x) - w_+(y)\phi^p(y))
\\[1ex] \notag &&\qquad \quad  \geq 
(w_+(x)-w_+(y))^{p-1} (w_+(x) (\max\{\phi(x),\phi(y)\})^p - w_+(y)\phi^p(y))
\\ \notag & & \qquad \quad \quad - \frac12 \,(w_+(x)-w_+(y))^p (\max\{\phi(x),\phi(y)\})^p
\\ \notag & & \qquad \quad \quad - c(\max\{w_+(x),w_+(y)\})^p |\phi(x)-\phi(y)|^p
\\[1ex] \notag &&\qquad \quad  \geq 
\frac12 \,(w_+(x)-w_+(y))^p( \max\{\phi(x),\phi(y)\})^p 
\\ \notag & & \qquad \quad \quad -  c(\max\{w_+(x),w_+(y)\})^p |\phi(x)-\phi(y)|^p
\end{eqnarray*}
whenever $w_+(x) \geq  w_+(y)$.
If, on the other hand, $w_+(y) > w_+(x)$ in the integrand, we may interchange the roles of $x$ and $y$ in the  display above by analogous reasoning. Hence we arrive in all cases at
\begin{eqnarray}
&& \notag \!\!\!\!\!\!\!\!  \int_{B_r}\int_{B_r} K(x,y)|w_+(x)-w_+(y)|^{p-2} \\
&& \notag\qquad \qquad   \times (w_+(x)-w_+(y))(w_+(x)\phi^p(x)-w_+(y)\phi^p(y))\,{\rm d}x{\rm d}y \nonumber \\[1ex]
&& \label{eq111112} \qquad \geq \ \frac12  \int_{B_r}\int_{B_r} K(x,y)|w_+(x)-w_+(y)|^{p}(\max\{\phi(x),\phi(y)\})^p\,{\rm d}x{\rm d}y\\[1ex]
&& \notag \qquad  \quad \ - c \int_{B_r}\int_{B_r} K(x,y)(\max\{w_+(x),w_+(y)\})^p |\phi(x)-\phi(y)|^p\,{\rm d}x{\rm d}y.
\end{eqnarray}
Observing finally that
\begin{eqnarray*} 
\notag
 |w_+(x)\phi(x) - w_+(y)\phi(y)|^p 
   & \leq & 2^{p-1} |w_+(x)-w_+(y)|^p (\max\{\phi(x),\phi(y)\})^p 
\\ && \notag  + 2^{p-1}  (\max\{w_+(x),w_+(y)\})^p |\phi(x)-\phi(y)|^p
\end{eqnarray*}
and combining this with~\eqref{eq111111} and~\eqref{eq111112} concludes the proof of~\eqref{cacio1} for $w_+$. 

In order to prove the estimate in~\eqref{cacio1} for $w_-$, it will suffice  to proceed as above, using the function $\eta=-w_-\,\phi$, instead of $\eta=w_+\,\phi$, as a test function in the weak formulation of problem \eqref{pb1}. \end{proof}

Above we made use of the following trivial but very useful small lemma.
\begin{lemma} \label{lemma:trivial conv} Let $p\geq 1$ and $\eps \in (0,1]$. Then
\[
|a|^p \leq |b|^p + c_p \eps  |b|^p + (1+c_p\eps)\eps^{1-p} |a-b|^p , \qquad c_p := (p-1) \Gamma(\max\{1,p-2\}) , 
\]
holds for every $a,b \in \mathds{R}^m$, $m \geq 1$. Here $\Gamma$ stands for the standard Gamma function.
\end{lemma}
\begin{proof}
By the triangle inequality and convexity we obtain 
\begin{eqnarray*} 
\notag |a|^p & \leq & (|b| + |a-b|)^p 
\\Ê\notag & = &  (1+\eps)^p \left(\frac{1}{1+\eps} |b| + \frac{\eps}{1+\eps}  \frac{|a-b|}{\eps} \right)^p
\\Ê\notag & \leq & (1+\eps)^{p-1} |b|^p + \left(\frac{1+\eps}{\eps}\right)^{p-1} |a-b|^p .
\end{eqnarray*}
Estimating 
\[
(1+\eps)^{p-1} = 1 + (p-1)\int_{1}^{1+\eps} t^{p-2} \, dt \leq 1 + 
\eps(p-1) \max\{1,(1+\eps)^{p-2}\},
\]
and then iterating, to get the Gamma function bound, concludes the proof.
\end{proof}

We would like to recall that, as in the classic local case,  the proven {\it Caccioppoli estimates with tail}\, encode basically all the informations deriving from the minimum property of the functions $u$ for what concerns the corresponding H\"older continuity.

\vspace{3mm}

We next show the validity of the second main tool, that is the {\it Logarithmic Lemma}~\ref{log_lemma}.

\begin{proof}[Proof of Log {Lemma~\ref{log_lemma}}]
Let $d>0$ be a real parameter and let $\phi \in C^\infty_0(B_{3r/2})$ be such that
$$
0\leq \phi\leq 1, \quad \phi\equiv 1 \  \text{in}\  B_{r} \quad\text{and}\quad |D\phi|<c\,r^{-1} \ \text{in} \ B_r \subset B_{R/2}.
$$
We use in the weak formulation of  problem \eqref{pb1}, the test function~$\eta$ defined by
$$
\eta=(u+d)^{1-p}\phi^p.
$$
Note that since $u \geq 0$ in the support of $\phi$, the test function is well-defined. 
We get
\begin{eqnarray}\label{6star}
{\displaystyle}
0 & = & \int_{B_{2r}}\int_{B_{2r}}K(x,y)|u(x)-u(y)|^{p-2}(u(x)-u(y))
\nonumber \\
&&  \qquad  \qquad \times\left[\frac{\phi^p(x)}{(u(x)+d)^{p-1}}-\frac{\phi^p(y)}{(u(y)+d)^{p-1}}\right]{\rm d}x{\rm d}y \nonumber \\ \nonumber
&& + 2 \int_{\mathds{R}^n\setminus B_{2r}}\int_{B_{2r}}K(x,y)|u(x)-u(y)|^{p-2}\frac{u(x)-u(y)}{(u(x)+d)^{p-1}}\phi^p(x)\,{\rm d}x{\rm d}y\\
&=:& I_1+I_2. 
\end{eqnarray}
If $u(x) > u(y)$, for the integrand of $I_1$,  we use the inequality in Lemma~\ref{lemma:trivial conv}, by choosing there $a=\phi(x)$, $b=\phi(y)$ and
$$
\varepsilon=\delta\,\frac{u(x)-u(y)}{u(x)+d}\in(0,1),\quad \text{with } \delta\in(0,1),
$$ 
since $u(y)\geq 0$ for any $y\in B_{2r}\subset B_R$. 
It follows that
\begin{eqnarray}\label{eq_13gat}
&& \hspace{-0.5cm} K(x,y)|u(x)-u(y)|^{p-2}(u(x)-u(y))\left[\frac{\phi^p(x)}{(u(x)+d)^{p-1}}-\frac{\phi^p(y)}{(u(y)+d)^{p-1}}\right]\nonumber\\[0.8ex]
&&  \quad   \quad\leq  K(x,y)\frac{(u(x)-u(y))^{p-1}}{(u(x)+d)^{p-1}}\phi^p(y)\left[1+c\,\delta\frac{u(x)-u(y)}{u(x)+d}-\left(\frac{u(x)+d}{u(y)+d}\right)^{p-1}\right]\nonumber\\
&&  \quad  \quad  \quad+ \,c\, \delta^{1-p}K(x,y)|\phi(x)-\phi(y)|^p,\nonumber
\end{eqnarray}
where $c\equiv c(p)$.
Observe that the first term that appears in the right-hand side of the previous inequality can be rewritten as
\begin{equation}\label{eq_13gat2}
K(x,y)\left(\frac{u(x)-u(y)}{u(x)+d}\right)^p\phi^p(y)\left[\frac{1-\left(\frac{u(y)+d}{u(x)+d}\right)^{1-p}}{1-\frac{u(y)+d}{u(x)+d}}+c\,\delta\right]=:J_1.
\end{equation} 

Now, consider the real function $t \mapsto g(t)$ given by
$$
{\displaystyle}
g(t):=\frac{1-t^{1-p}}{1-t}=-\frac{p-1}{1-t}\int_{t}^{1}\tau^{-p}\,{\rm d}\tau, \quad \forall\, t\in(0,1).
$$
We have that $g$ is an increasing function in $t$, since
$$
t\mapsto \frac{1}{1-t}\int_{t}^{1}\tau^{-p}\,{\rm d}\tau
$$ 
is a decreasing function (recall that $p>1$). Thus
$$
g(t)\leq -(p-1)\quad \forall \, t\in(0,1).
$$ 
Moreover, if $t\leq 1/2$, then
$$
g(t)\leq -\frac{p-1}{2^{p}}\, \frac{t^{1-p}}{1-t}.
$$
Therefore, if 
$$
t=\frac{u(y)+d}{u(x)+d}\in (0,\, 1/2]; 
$$
that is, 
$$
u(y)+d\leq \frac{u(x)+d}{2},
$$
then, since
$
{(u(x)-u(y))(u(y)+d)^{p-1}}/{(u(x)+d)^p}\leq 1$, we get
\begin{equation}\label{eq_14gatbis}
J_1
\, \leq \, K(x,y)\left(c\,\delta-\frac{p-1}{2^{p}}\right)\left[\frac{u(x)-u(y)}{u(y)+d}\right]^{p-1} \phi^p(y),
\end{equation}
Choosing
\begin{equation}\label{def_delta}
\delta=\frac{p-1}{2^{p+1}\,c},
\end{equation}
we get
\begin{equation*}
J_1\leq - K(x,y)\,\frac{p-1}{2^{p+1}}\left[\frac{u(x)-u(y)}{u(y)+d}\right]^{p-1}
\end{equation*}
If, on the other hand, 
$$
t=\frac{u(y)+d}{u(x)+d}\in(1/2,\, 1),
$$
that is,
$$
u(y)+d> \frac{u(x)+d}{2},
$$
then
$$
J_1
\, \leq\, K(x,y)\,[c\delta-(p-1)]\left[\frac{u(x)-u(y)}{u(x)+d}\right]^{p} \phi^p(y),
$$\\
and so, by the choice of $\delta$ in \eqref{def_delta}, we finally get
\begin{equation}\label{eq_14gat}
J_1\leq -K(x,y)\,\frac{(2^{p+1}-1)(p-1)}{2^{p+1}}\left[\frac{u(x)-u(y)}{u(x)+d}\right]^{p}\!\phi^p(y).
\end{equation} 
Furthermore, if $2(u(y)+d)<u(x)+d$, then 
\begin{equation}\label{eq_14gat2}
\left[\log\left(\frac{u(x)+d}{u(y)+d}\right)\right]^p
\,\leq\, c \left[\frac{u(x)-u(y)}{u(y)+d}\right]^{p-1}
\end{equation} 
holds with $c \equiv c(p)$. On the other hand, if $2(u(y)+d)\geq u(x)+d$, recalling that we have assumed $u(x) > u(y)$, then
\begin{equation}\label{eq_14gat3}
\left[\log\left(\frac{u(x)+d}{u(y)+d}\right)\right]^p
\, = \,\left[\log\left(1+ \frac{u(x)-u(y)}{u(y)+d}\right)\right]^{p}
\,\leq \, 2^p\left(\frac{u(x)-u(y)}{u(x)+d}\right)^{p},
\end{equation}
where we have used
$$
\log(1+\xi)\leq \xi, \quad \forall\,\, \xi\geq 0, \quad \text{with}\,\,\, \xi=\frac{u(x)-u(y)}{u(y)+d}\leq \frac{2[u(x)-u(y)]}{u(x)+d}.
$$ 

Thus, combining~\eqref{eq_13gat} with \eqref{eq_14gatbis}, \eqref{eq_14gat}, \eqref{eq_14gat2} and~\eqref{eq_14gat3},
we conclude with
\begin{eqnarray*}
&& \hspace{-0.5cm} K(x,y)|u(x)-u(y)|^{p-2}(u(x)-u(y))\left[\frac{\phi^p(x)}{(u(x)+d)^{p-1}}-\frac{\phi^p(y)}{(u(y)+d)^{p-1}}\right]\nonumber\\
&&  \quad   \quad\leq \ - \frac{1}{c} \,  K(x,y)\left[\log\left(\frac{u(x)+d}{u(y)+d}\right)\right]^p \phi^p(y) + \,c\, \delta^{1-p}K(x,y)|\phi(x)-\phi(y)|^p.\nonumber
\end{eqnarray*}
Observe that when $u(x)=u(y)$, then the  estimate above holds trivially. If, on the other hand, $u(y)>u(x)$ we can again exchange the roles of $x$ and $y$ in the  computations above. We finally get for the first term in~\eqref{6star} that
\begin{eqnarray}\label{I1}
I_1 &\leq& -\frac{1}{c}\int_{B_{2r}}\int_{B_{2r}} K(x,y)\left|\log\left(\frac{u(x)+d}{u(y)+d}\right)\right|^p \phi^p(y)\,{\rm d}x{\rm d}y\\
\nonumber&&+c\int_{B_{2r}}\int_{B_{2r}} K(x,y)|\phi(x)-\phi(y)|^p\,{\rm d}x{\rm d}y
\end{eqnarray}
for a constant $c\equiv c(p)$ by the choice of $\delta$.

For the second contribution in~\eqref{6star}, namely $I_2$, we can proceed as follows.
First of all, notice that when $y\in B_R$, $u(y)\geq 0$ and so
$$
\frac{(u(x)-u(y))_+^{p-1}}{(d+u(x))^{p-1}}\leq 1\quad \text{ for all } x \in B_{2r}, \; y \in B_R.
$$
Moreover, when $y\in \mathds{R}^n\setminus B_R$,
$$
(u(x)-u(y))_+^{p-1}\leq 2^{p-1}[u^{p-1}(x)+(u(y))_-^{p-1}]\quad \text{ for all }  x \in B_{2r}.
$$
Therefore,
\begin{eqnarray}\label{eq_14fs}
I_2& \leq & 2 \int_{B_R\setminus B_{2r}}\int_{B_{2r}} K(x,y)(u(x)-u(y))_+^{p-1}(d+u(x))^{1-p}\,\phi^p(x)\,{\rm d}x{\rm d}y
\nonumber\\
&& +2 \int_{\mathds{R}^n\setminus B_R}\int_{B_{2r}} K(x,y)(u(x)-u(y))_+^{p-1}(d+u(x))^{1-p} \phi^p(x)\,{\rm d}x{\rm d}y
\nonumber 
\\[1ex]
&\leq &c  \int_{\mathds{R}^n\setminus B_{2r}}\int_{B_{2r}} K(x,y)\phi^p(x)\,{\rm d}x{\rm d}y  \\
&& \quad + c d^{1-p}\, \int_{\mathds{R}^n\setminus B_R}\int_{B_{2r}}  K(x,y)(u(y))_-^{p-1}\,{\rm d}x{\rm d}y\nonumber
\end{eqnarray}
follows for $c\equiv c(p)$.
By the assumptions on $K$ and the fact that the support of $\phi$ belongs to $B_{3r/2}$, we have
\begin{equation}\label{eq_14fs2}
\int_{\mathds{R}^n\setminus B_{2r}}\int_{B_{2r}}K(x,y)\phi^p(x)\,{\rm d}x{\rm d}y 
\, \leq \, c\sup_{x\in B_{3r/2}}r^n\int_{\mathds{R}^n\setminus B_{2r}} K(x,y)\,{\rm d}y \leq c r^{n-sp}
\end{equation}
and
\begin{eqnarray}\label{eq_15fs}
 \int_{\mathds{R}^n\setminus B_R}\int_{B_{2r}} K(x,y)(u(y))_-^{p-1}\,{\rm d}x{\rm d}y &\leq &  c \,|B_r|  \int_{\mathds{R}^n\setminus B_R}\frac{(u(y))_-^{p-1}}{|y-x_0|^{n+sp}}\,{\rm d}y \nonumber \\[1ex]
&\leq &c\, \frac{r^n}{R^{sp}}\,\left[{\rm Tail}(u_-;x_0,R)\right]^{p-1},
\end{eqnarray}
where we also used that, for any $x\in B_r$, $y\in \mathds{R}^n \setminus B_R$ and $2r\leq R$,
$$
\frac{|y-x_0|}{|y-x|}\leq 1+\frac{|x-x_0|}{|x-y|}\leq 1+\frac{r}{R-r}\leq 2.
$$ 
By combining~\eqref{eq_14fs} with~\eqref{eq_14fs2} and~\eqref{eq_15fs}, we  obtain
\begin{eqnarray}\label{eq_15fs2}
{\displaystyle}
I_2& \leq & 
c\,\int_{B_{2r}}\int_{B_{2r}} K(x,y)|\phi(x)-\phi(y)|^p\,{\rm d}x{\rm d}y + c r^{n-sp} \nonumber\\
&& +c\,d^{1-p}\,r^{n}\, R^{-sp}\,\left[{\rm Tail}(u_-;x_0,R)\right]^{p-1},
\nonumber
\end{eqnarray}
which, together with~\eqref{I1} in \eqref{6star}, yields
 \begin{eqnarray}\label{est1}
&& \nonumber  \hspace{-0.5cm} \int_{B_{2r}}\int_{B_{2r}} K(x,y)\left|\log\left(\frac{u(x)+d}{u(y)+d}\right)\right|^p\phi^p(y)\,{\rm d}x{\rm d}y \\[1ex]
&& \qquad \qquad \qquad \leq \, c\,\int_{B_{2r}}\int_{B_{2r}} K(x,y)|\phi(x)-\phi(y)|^p\,{\rm d}x{\rm d}y
\\[1ex]
\nonumber 
&& \qquad \qquad \qquad \quad +c\,d^{1-p}\,r^{n}\, R^{-sp}\,\left[{\rm Tail}(u_-;x_0,R)\right]^{p-1} +  c r^{n-sp}.
\end{eqnarray}
 
Finally, in order to conclude the proof, we need the following estimate
\begin{eqnarray}\label{phi}
\nonumber \int_{B_{2r}}\int_{B_{2r}}K(x,y)|\phi(x)-\phi(y)|^p\,{\rm d}x{\rm d}y & \leq &cr^{-p} \int_{B_{2r}}\int_{B_{2r}}| x-y|^{-n+p(1-s)} \,{\rm d}x{\rm d}y \\  
&\leq& \frac{c}{p(1-s)}\, r^{-sp} |B_{2r}|,
\end{eqnarray}
where we used the bound from above on the kernel $K$ and the fact that we are assuming $|D \phi|\leq c\, r^{-1}$.
The proof of \eqref{log_estimate} is finished.
\end{proof}

A first consequence of the Logarithmic Lemma is the following
\begin{corollary}\label{poincare+cacio}
Let $p\in (1,\infty)$ and let $u\in W^{s,p}(\mathds{R}^n)$ be the solution to  problem \eqref{pb1} such that $u\geq 0$ in $B_R \equiv B_R(x_0)  \subset \Omega$. Let $a,d>0$, $b>1$ and define
\[
v := \min\left\{\left(\log(a+d) - \log(u+d)  \right)_+, \, \log(b) \right\}.
\]
Then the following estimate holds true, for any $B_r\equiv B_r(x_0)\subset B_{R/2}(x_0)$,   
\begin{eqnarray}\label{log_estimate2}
 \dashint_{B_{r}} |v-(v)_{B_{r}}|^p\,{\rm d}x \nonumber  
& \leq & c  \left\{d^{1-p}\,\left(\frac{r}{R}\right)^{sp}\,\left[{\rm Tail}(u_-;x_0,R)\right]^{p-1}+1\right\}\!,
\end{eqnarray}
where ${\rm Tail}(u_-;x_0,R)$ is defined by~\eqref{Tail} and $c$ depends only on $n$, $p$, $s$, $\lambda$ and $\Lambda$. 
\end{corollary}
\begin{proof}
By the fractional Poincar\'e type inequality (see, e.~\!g., Proposition 5.1, Formula~(6.3) in \cite{Min03}) and the assumption in~\eqref{hp_K} for $K$ we get
\begin{equation*}
\dashint_{B_{r}}|v-(v)_{B_{r}}|^p\,{\rm d}x 
\, \leq \,
c\, r^{sp-n} \int_{B_{r}}\int_{B_{r}}  K(x,y)|v(x)-v(y)|^p\,{\rm d}x{\rm d}y 
\end{equation*}
with a constant $c\equiv c(n,p,s,\lambda,\Lambda)$. Now observe that $v$ is a truncation of the sum of a constant and $\log(u+d)$ and hence it follows that 
\begin{eqnarray*}
\int_{B_{r}}\int_{B_{r}}K(x,y) |v(x)-v(y)|^p \dxy
\, \leq\,
 \int_{B_{r}}\int_{B_{r}} K(x,y)\left| \log\left( \frac{u(y)+d}{u(x)+d} \right) \right|^p  \dxy. 
\end{eqnarray*}
At this stage, in order to conclude, it just suffices to apply the estimate in~\eqref{log_estimate}. 
\end{proof}

\vspace{3mm}

\section{Local boundedness}\label{sec_sup}
In this section, we prove the local boundedness for the fractional $p$-minimizers of the functional \eqref{def_F}, as stated in Theorem~\ref{sup}.
\begin{proof}[Proof of {Theorem~\ref{sup}}]
Before starting, let us give some definitions. For any $j\in \mathds{N}$ and $r>0$ such that $B_r(x_0)\subset\Omega$,
\begin{equation}\label{raggi}
r_j=\frac 12(1+2^{-j})r,\quad \tilde{r}_j=\frac{r_j+r_{j+1}}{2}, 
\end{equation}
\begin{equation*}
B_j=B_{r_j}(x_0),\quad \tilde{B}_j=B_{\tilde{r}_j}(x_0).
\end{equation*}
Moreover, take
\begin{equation*}
\phi_j\in C_0^\infty(\tilde{B}_j), \,\,\, 0\leq \phi_j\leq 1,\,\,\,  \phi_j\equiv 1\text{ on } B_{j+1}, \,\,\,\text{and} \,\,\, |D\phi_j|<2^{j+3}/r, 
\end{equation*}
\begin{equation*} 
k_j=k+(1-2^{-j})\tilde{k}, \quad \tilde{k}_j=\frac{k_{j+1}+k_j}{2}, \quad \tilde{k}\in \mathds{R}^+ \text{ and } k\in \mathds{R}.
\end{equation*}
\begin{equation}\label{w_tilde}
\tilde{w}_j=(u-\tilde{k}_j)_+\quad\text{and}\quad w_j=(u-k_j)_+.
\end{equation}

By the fractional Poincar\'e inequality applied to the function $\tilde{w}_j\phi_j$, as defined above, together with the properties of the kernel $K$, we plainly get
\begin{eqnarray}\label{riviere}
&&  \left| \left(\dashint_{B_j}|\tilde{w}_j(x)\phi_j(x)|^{p^*}\,{\rm d}x\right)^{\!\frac{1}{p^*}} \!
 - \left|\dashint_{B_j} \tilde{w}_j(x)\phi_j(x)\,{\rm d}x\right| \right|^p \nonumber \\
 &&\qquad \qquad  \leq \,
c \,\left( \dashint_{B_j}|\tilde{w}_j\phi_j - (\tilde{w}_j\phi_j)_{B_j}|^{p^*}\,{\rm d}x\right)^{\!\frac{p}{p^*}} \\
&& \qquad \qquad \leq  \,
 c\, \frac{r^{sp}}{r^n}\int_{B_j}\int_{B_j}K(x,y)|\tilde{w}_j(x)\phi_j(x)-\tilde{w}_j(y)\phi_j(y)|^p \,{\rm d}x{\rm d}y, \nonumber
\end{eqnarray}
where $p^*=np/(n-sp)$ is the critical exponent for fractional Sobolev embeddings, so that we are now dealing with the case when $sp<n$.

Using the nonlocal Caccioppoli inequality with tail given by~\eqref{cacio1}, with $w_+=\tilde{w}_j$ and $\phi=\phi_j$ there, we arrive at 
\begin{eqnarray}\label{sob+cacio}
\nonumber
&& \left(\dashint_{B_j}|\tilde{w}_j(x)\phi_j(x)|^{p^*}\,{\rm d}x\right)^{\frac{p}{p^*}} \\[1ex]
&& \qquad \qquad \quad \leq  c r^{sp}\dashint_{B_j}\int_{B_j}K(x,y)(\max\{\tilde{w}_j(x), \tilde{w}_j(y)\})^p|\phi_j(x)-\phi_j(y)|^p \,{\rm d}x{\rm d}y\\[1ex]
&& \qquad \qquad\qquad   + \,c r^{sp}\dashint_{B_j}\tilde{w}_j(y)\phi^p_j(y) \,{\rm d}y\left(\sup_{y\in{\rm supp}\,\phi_j}\int_{\mathds{R}^n\setminus B_j}K(x,y)\tilde{w}_j^{p-1}(x)\,{\rm d}x\right) \nonumber
\\
&& \qquad \qquad \qquad 
+ \,c \,\dashint_{B_j} |\tilde{w}_j(x)\phi_j(x)|^{p}\,{\rm d}x.
 \nonumber
\end{eqnarray}

By the definition of $\phi_j$ and the assumption \eqref{hp_K}, we obtain the following estimate for the first term in the right hand-side of the  inequality above,
\begin{eqnarray}\label{10star}
&& \hspace{-1cm} r^{sp}\dashint_{B_j}\int_{B_j} K(x,y)(\max\{\tilde{w}_j(x), \tilde{w}_j(y)\})^p|\phi_j(x)-\phi_j(y)|^p\,{\rm d}x{\rm d}y 
\nonumber \\ [1ex]
&& \qquad\qquad\qquad \qquad \leq \, 
c  2^{jp} \frac{ r^{sp}}{r^p}\,\dashint_{B_j}w_j^p(y)\left(\int_{B_j}\frac{{\rm d}x}{|x-y|^{n-p(1-s)}}\right){\rm d}y\nonumber \\ 
&& \qquad\qquad \qquad\qquad \leq\, \frac{c 2^{jp}}{p(1-s)}\,\dashint_{B_j}w_j^p(x)\,{\rm d}x.
\end{eqnarray}

For the second term on the right in \eqref{sob+cacio}, we get
\begin{eqnarray}\label{10star2}
&& 
c\, r^{sp}\dashint_{B_j}\tilde{w}_j(y)\phi^p_j(y) \,{\rm d}y\left(\sup_{y\in{\rm supp}\,\phi_j}\int_{\mathds{R}^n\setminus B_j}K(x,y)\tilde{w}_j^{p-1}(x)\,{\rm d}x\right)
\nonumber \\[1ex]
&&\qquad\qquad  \qquad \leq c\, 2^{j(n+sp)}r^{sp}\left(\dashint_{B_j}\frac{w_j^p(y)}{(\tilde{k}_j-k_j)^{p-1}}\,{\rm d}y\right)\left(\int_{\mathds{R}^n\setminus B_j}\frac{w_j^{p-1}(x)}{|x_0-x|^{n+sp}}\,{\rm d}x\right) \nonumber \\[1ex]
&&\qquad\qquad  \qquad \leq c\, \frac{2^{j(n+sp+p-1)}}{\tilde{k}^{p-1}}\,\left[{\rm Tail}(w_0;x_0,r/2)\right]^{p-1}\, \dashint_{B_j}w_j^p(y)\,{\rm d}y,
\end{eqnarray}
where we have just used the definitions in \eqref{raggi}--\eqref{w_tilde}, the facts that $\tilde{w}_j\leq w_j^p/(\tilde{k}_j-k_j)^{p-1}$ and that $y\in \tilde{B}_j={\rm supp}\,\phi_j$ and $x\in \mathds{R}^n\setminus B_j$ yield
$$
\frac{|x-x_0|}{|x-y|}\,\leq\, \frac{|x-y|+|x_0-x|}{|x-y|}\,\leq \,1+\frac{\tilde{r}_j}{r_j-\tilde{r}_j}\,\leq\,  2^{j+4}.
$$

The left hand-side of \eqref{sob+cacio} can be estimated from below as follows
\begin{eqnarray}\label{10star3}
\left(\dashint_{B_j}|\tilde{w}_j(x)\phi_j(x)|^{p^*}\,{\rm d}x\right)^{\frac{p}{p^*}}
& \geq & (k_{j+1}-\tilde{k}_j)^{\frac{(p^*-p)p}{p^*}}\left(\dashint_{B_{j+1}} w_{j+1}^p(x)\,{\rm d}x\right)^{\frac{p}{p^*}}  \nonumber \\[1ex]
& = & \left(\frac{\tilde{k}}{2^{j+2}}\right)^{\frac{(p^*-p)p}{p^*}} \left(\dashint_{B_{j+1}} w_{j+1}^p(x)\,{\rm d}x\right)^{\frac{p}{p^*}}.
\end{eqnarray}

By combining~\eqref{sob+cacio} with~\eqref{10star}, \eqref{10star2} and \eqref{10star3},
we obtain
\begin{equation*}
\left(\frac{\tilde{k}^{1-{p}/{p^*}}}{2^{(j+2)\frac{(p^*-p)}{p^*}}}\right)^p
A_{j+1}^{\frac{p^2}{p^*}}
\, \leq \, 
c\, 2^{j(n+sp+p-1)}\left(\frac{1}{p(1-s)}+\frac{\left[{\rm Tail}(w_0;x_0,r/2)\right]^{p-1}}{\tilde{k}^{p-1}}
+1
\right)A_j^p,
\end{equation*}
where we have set 
$
{\displaystyle} A_j:= \left(\dashint_{B_{j}} w_{j}^p(x)\,{\rm d}x\right)^{\frac{1}{p}}.
$

Now, by taking
\begin{equation}\label{18se}
\tilde{k}\geq \delta {\rm Tail}(w_0;x_0,r/2), \qquad \delta \in (0,1],
\end{equation}
we get
\begin{equation}\label{18se2}
{\displaystyle}
\left(\frac{A_{j+1}}{\tilde{k}}\right)^{\!\frac{p}{p^\ast}} 
 \leq \,
\delta^{\frac{1-p}{p} } \bar{c}^{\frac{p}{p^\ast}} 2^{j\left(\frac{n+sp+p-1}{p}+\frac{sp}{n}\right)}\frac{A_j}{\tilde{k}},
\end{equation}
where
$\bar{c}=2^{\frac{2(p^\ast-p)}{p}} c^{\frac{p^\ast}{p^2}} {(2+(p(1-s))^{-1})}^{\frac{p^\ast}{p^2}} .$

Setting $\beta:=sp/(n-sp) = p^\ast/p-1>0$ and $C:= 2^{\frac{(n+sp+p-1)n}{p(n-sp)}+\frac{sp}{n-sp}}>1$, the estimate in~\eqref{18se2} becomes
$$
{\displaystyle}
\frac{A_{j+1}}{\tilde{k}} 
\, \leq \,
 \delta^{\frac{(1-p)p^\ast}{p^2} }\bar{c}\, C^j \left( \frac{A_j}{\tilde{k}} \right)^{\!1+\beta}
$$

Thus, it suffices to prove that the following estimate on $A_0$ does hold
\begin{equation*}
\frac{A_0}{\tilde{k}}
\, \leq \, \delta^{\frac{(p-1)p^\ast}{\beta p^2} } \bar{c}^{\!-\frac{1}{\beta}} C^{\!-\frac{1}{\beta^2}}
\end{equation*}
and, by a well-known iteration argument, it will follow
$A_j\rightarrow 0$ as $j\rightarrow \infty$.
Since
\[
\frac{(p-1)p^\ast}{\beta\, p^2} = \frac{p-1}{p} \frac{n}{n-ps} \frac{n-sp}{sp} = \frac{(p-1)n}{s p^2},
\]
we choose
\begin{equation*} 
\tilde{k}=\delta \, {\rm Tail}(w_0;x_0,r/2)+ \delta^{-\frac{(p-1)n}{s p^2}}H A_0,\quad  \text{with $H:=\bar{c}^{\frac{1}{\beta}} C^{\frac{1}{\beta^2}}$},
\end{equation*}
which is in accordance with~\eqref{18se}.

We deduce
\begin{eqnarray*}
\sup_{B_{r/2}} u & \leq & k+\tilde{k} \\[-4ex]
& = & k+\delta {\rm Tail}((u-k)_+;x_0,r/2)+\delta^{-\frac{(p-1)n}{s p^2}} H\left(\dashint_{B_r}(u-k)_+^p\right)^{\frac 1p},
\end{eqnarray*}
which finally gives the desired result by taking $k=0$.

 The remaining case, that is when $sp=n$, can be treated exactly as above, just replacing $p^\star$ by a suitable power~$q$ in the left hand-side of~\eqref{riviere} and consequently adjusting the exponents in the rest of the proof.
\end{proof}

\begin{rem}\label{rem_inf}
Similarly, it is possible to prove that the weak solutions to problem~\eqref{pb1} are locally bounded from below, satisfying an estimate analogous to the one in~\eqref{sup_estimate}. The proof is exactly as before: one has just to work with
$\tilde{w}_j=(\tilde{k}_j-u)_+$ and $w_j=(k_j-u)_+$ instead of the auxiliary functions defined in~\eqref{w_tilde} and make use of the corresponding Caccioppoli estimate \eqref{cacio1} for $w_-$.
\end{rem}

\vspace{3mm}
\section{H\"older continuity}\label{sec_holder}

This section is devoted to the proof of the H\"older continuity of solutions, namely Theorem~\ref{holder1}.
As in the local framework, an iteration lemma is the keypoint of the proof. However, as before, we have to handle the nonlocality of the involved operator and thus a certain care is required. In the proof below, all the estimates proven in previous sections will appear.

Before starting, let us fix some notation. For any $j\in \mathds{N}$, let $0<r<R/2$, for some $R$ such that $B_R(x_0)\subset\Omega$, 
$$
r_j:=\sigma^j\,\frac r2,\quad \sigma\in\left(0,1/4\right]\quad \text{and} \quad
B_j:=B_{r_j}(x_0).
$$

Moreover, let us define 
\begin{equation*}
\frac 12\, \omega(r_0)=\frac 12 \,\omega \left(\frac r 2\right):={\rm Tail}(u;x_0,r/2)+c\left(\dashint_{B_r}|u|^p\,{\rm d}x\right)^{\frac 1p},
\end{equation*}
with ${\rm Tail}(u;x_0,r/2)$ as in \eqref{Tail} and $c$ as in \eqref{sup_estimate},  and 
\begin{equation*}
\omega(r_j):=\left(\frac{r_j}{r_0}\right)^{\alpha}\,\omega(r_0),\quad\text{for some } \alpha<\frac{sp}{p-1}.
\end{equation*}

In order to prove the Theorem~\ref{holder1}, it will suffice to prove the following
\begin{lemma}\label{iteration}
Under the notation introduced above, let $u\in W^{s,p}(\mathds{R}^n)$ be the solution to  problem \eqref{pb1}. Then \begin{equation}\label{iter1}
\osc_{B_j} u\, \equiv \, \sup_{B_j} u-\inf_{B_j} u
\, \leq \, \omega(r_j),\quad \forall\, j=0,1,2,....
\end{equation}
\end{lemma}
\begin{proof}
We will proceed by induction. 
For this,  note that by the definition of $\omega(r_0)$ and Theorem~\ref{sup} (with $\delta=1$ there), the estimate in~\eqref{iter1} trivially holds for $j=0$, since, in particular, both the functions $(u)_+$ and $(-u)_+$ are weak subsolutions.

Now, we make a strong induction assumption and assume that \eqref{iter1} is valid for all $i \in \{0,\ldots , j\}$ for some $j \geq 0$, and then we prove that it  holds also for $j+1$. We have that either
\begin{equation}\label{hp_density1}
\frac{|2 B_{j+1}\cap \{u\geq \inf_{B_j} u + \omega(r_j)/2\}|}{|2 B_{j+1}|} \, \geq \, \frac12.
\end{equation}
or 
\begin{equation}\label{hp_density2}
\frac{|2 B_{j+1}\cap \{u \leq \inf_{B_j} u + \omega(r_j)/2\}|}{|2 B_{j+1}|} \, \geq \, \frac12
\end{equation} 
must hold.
If~\eqref{hp_density1} holds, we set $u_j := u - \inf_{B_j} u $, and if~\eqref{hp_density2} holds, we set $ u_j := \omega(r_j)-(u-\inf_{B_j} u)$.
In all cases we have that $u_j \geq 0$ in $B_j$ and 
\begin{equation}\label{hp_density}
\frac{|2 B_{j+1}\cap \{u_j \geq \omega(r_j)/2\}|}{|2 B_{j+1}|} \, \geq \, \frac12
\end{equation}
holds. Moreover, $u_j$ is a weak solution satisfying
\begin{equation} \label{eq:ind cor}
\sup_{B_i} |u_j| \leq 2 \omega(r_i) \qquad \forall \,\,i \, \in \{0,\ldots , j\}.
\end{equation}

We now claim that under the induction assumption we have
\begin{equation}\label{step1}
\left[{\rm Tail}(u_j;x_0,r_j)\right]^{p-1}\leq c\, \sigma^{-\alpha(p-1)}[\omega(r_j)]^{p-1},
\end{equation}
where the constant $c$ depends only on $n,p,s$ and the difference of $sp/(p-1)$ and $\alpha$, but, in particular, it is independent of $\sigma$. 
Indeed, we have
\begin{eqnarray*}\label{tail 1}
\left[{\rm Tail}(u_j; x_0, r_j)\right]^{p-1}& =&  r_j^{sp}\sum_{i=1}^j\int_{B_{i-1}\setminus B_i}|u_j(x)|^{p-1}|x-x_0|^{-n-sp}\,{\rm d}x \\
&& +\, r_j^{sp} \int_{\mathds{R}^n\setminus B_0} |u_j(x)|^{p-1}|x-x_0|^{-n-sp}\,{\rm d}x\\[1ex]
& \leq & r_j^{sp}\sum_{i=1}^j [\sup_{B_{i-1}} |u_j|]^{p-1}\int_{\mathds{R}^n\setminus B_i}|x-x_0|^{-n-sp}\,{\rm d}x\\
&& \quad  +\, r_j^{sp}\int_{\mathds{R}^n\setminus B_0} |u_j(x)|^{p-1}|x-x_0|^{-n-sp}\,{\rm d}x\\[1ex]
& \leq &  c\,\sum_{i=1}^j \left(\frac{r_j}{r_i}\right)^{sp} \,[\omega(r_{i-1})]^{p-1},
\end{eqnarray*}
where on the last line we used~\eqref{eq:ind cor} and 
\begin{eqnarray*} 
\notag
&& \hspace{-0.5cm}\int_{\mathds{R}^n\setminus B_0} |u_j(x)|^{p-1}|x-x_0|^{-n-sp}\,{\rm d}x 
\\ \notag && \quad \leq \ c r_0^{-sp}\sup_{B_0} |u|^{p-1} +  c r_0^{-sp}  [\omega(r_{0})]^{p-1}
+
c\int_{\mathds{R}^n\setminus B_0} |u(x)|^{p-1}|x-x_0|^{-n-sp}\,{\rm d}x
\\ \notag && \quad \leq c r_1^{-sp} [\omega(r_{0})]^{p-1}. \\
\end{eqnarray*}
Estimating further as
\begin{eqnarray*}\label{tail 2}
&& \sum_{i=1}^j \left(\frac{r_j}{r_i}\right)^{sp} \,[\omega(r_{i-1})]^{p-1} \\
&&\qquad = \  [\omega(r_0)]^{p-1}\left(\frac{r_j}{r_0}\right)^{\alpha(p-1)} \,\sum_{i=1}^j \left(\frac{r_{i-1}}{r_i}\right)^{\alpha(p-1)}\left(\frac{r_j}{r_i}\right)^{sp-\alpha(p-1)} \\
&& \qquad = \  [\omega(r_j)]^{p-1} \,\sigma^{-\alpha(p-1)}\,\sum_{i=0}^{j-1}\sigma^{i
(sp-\alpha(p-1))} \\[1ex]
&& \qquad \leq \  [\omega(r_j)]^{p-1} \,  \frac{\sigma^{-\alpha(p-1)}}{1-\sigma^{sp-\alpha(p-1)}  }
\\[1ex]
&&\qquad \leq \ \frac{4^{sp-\alpha(p-1)}}{\log(4)(sp-\alpha(p-1))} \, \sigma^{-\alpha(p-1)} \,[\omega(r_j)]^{p-1},
\end{eqnarray*}
where we have used the fact that $\sigma \leq 1/4$ and $\alpha<~sp/(p-1)$. 
Hence~\eqref{step1} is proved with  $c$ depending only on $n,p,s$ and the difference of $sp/(p-1)$ and $\alpha$.

Next, consider the function $v$ defined as follows
\begin{equation}\label{def_v}
v:=\min\left\{\left[\log\left(\frac{\omega(r_j)/2+d}{u_j+d}\right)\right]_+\!, \, k\right\}, \qquad k>0.
\end{equation}  
Applying  then Corollary~\ref{poincare+cacio}, obviously with $a \equiv \omega(r_j)/2$ and $b \equiv \exp(k)$, we get
$$
\dashint_{2B_{j+1}}|v-(v)_{2B_{j+1}}|^p\,{\rm d}x 
\, \leq \, c\, \left\{ d^{1-p} \left( \frac{r_{j+1}}{r_j} \right)^{sp}[{\rm Tail}(u_j;x_0, r_j)]^{p-1}+1\right\}.
$$
Thus, as a consequence of the estimate in \eqref{step1},
we arrive at
$$
\dashint_{2B_{j+1}}|v-(v)_{2B_{j+1}}|^p\,{\rm d}x
\, \leq \, c\left\{d^{1-p}\sigma^{sp-\alpha(p-1)}[\omega(r_j)]^{p-1}+1\right\}.
$$
Therefore, choosing $d=\varepsilon\, \omega(r_j)$ with
\begin{equation*}
\eps := \sigma^{\frac{sp}{p-1}-\alpha},
\end{equation*}
we get
\begin{equation}\label{est_v}
\dashint_{2B_{j+1}}|v-(v)_{2B_{j+1}}|\,{\rm d}x\leq c,
\end{equation}
where the constant $c$ depends only on $n,p,s,\lambda,\Lambda$ and the difference of $sp/(p-1)$ and $\alpha$. 

To continue, denote in short 
 $\tilde{B}\equiv 2B_{j+1}$, and 
follow the path paved in \cite[Lemma 2.107]{MZ97}, together with~\eqref{hp_density} and the definition of $v$ given in \eqref{def_v}. We obtain
\begin{eqnarray*}
k&=&\frac{1}{|\tilde{B}\cap \{u_j \geq \omega(r_j)/2\}|}\int_{\tilde{B}\,\cap\, \{u_j \geq \omega(r_j)/2\}} k\,{\rm d}x\\[1ex]
&=&\frac{1}{|\tilde{B}\cap \{u_j \geq \omega(r_j)/2\}|} \int_{\tilde{B}\,\cap\, \{v=0\}} k\,{\rm d}x\\[1ex]
&\leq &\frac{2}{|\tilde{B}|}\int_{\tilde{B}} (k-v)\,{\rm d}x
\ = \ 2[k-(v)_{\tilde{B}}].
\end{eqnarray*}
By integrating the preceding inequality over the set $\tilde{B}\cap\{v=k\}$ we obtain
\begin{eqnarray*}
\frac{|\tilde{B}\cap \{v=k\}|}{|\tilde{B}|} \,k&\leq& \frac{2}{|\tilde{B}|}\int_{\tilde{B}\,\cap\, \{v=k\}} [k-(v)_{\tilde{B}}]\,{\rm d}x\\[1ex]
&\leq &\frac{2}{|\tilde{B}|}\int_{\tilde{B}} |v-(v)_{\tilde{B}}|\,{\rm d}x\leq c,
\end{eqnarray*}
thanks to \eqref{est_v}. Let us take
$$
k=\log\left(\frac{\omega(r_j)/2+\varepsilon \,\omega(r_j)}{3\,\varepsilon\,\omega(r_j)}\right)=\log\left(\frac{1/2+\varepsilon}{3\,\varepsilon}\right)\approx \log \left(\frac 1\varepsilon\right),
$$ 
so that
$$
\frac{|\tilde{B}\cap \{v=k\}|}{|\tilde{B}|} \,k \, \leq \, c
$$
yields
\begin{equation}\label{MZ}
\frac{|\tilde{B}\cap \{u_j \leq 2\,\varepsilon\,\omega(r_j)\}|}{|\tilde{B}|}
\,  \leq \, \frac c k 
\leq  \frac{c_{\log}}{\log \left(\frac 1\sigma\right)},
\end{equation}
where the constant $c_{\log}$ depends only on $n,p,s,\lambda,\Lambda$ and the difference of $sp/(p-1)$ and $\alpha$ via the definition of $\eps$.

We are now in a position to start a suitable iteration to deduce the desired oscillation reduction. First, for any $i=0,1,2,...$, we define
\begin{equation*} 
\varrho_i=r_{j+1}+2^{-i}r_{j+1},\quad \tilde{\varrho}_i := \frac{\varrho_{i} +\varrho_{i+1}}{2}, \quad  B^i=B_{\varrho_i}, \quad \tilde B^i=B_{\tilde \varrho_i}
\end{equation*}
and corresponding cut-off functions
\begin{equation*}
\phi_i\in C_0^\infty(\tilde B^i), \,\,\, 0\leq \phi_i\leq 1,\,\,\,  \phi_i\equiv 1\text{ on } B^{i+1}, \,\,\,\text{and} \,\,\, |D\phi_i|<c\,\varrho_i^{-1}.
\end{equation*}
Furthemore, set
\begin{equation*} 
k_i=(1+2^{-i})\varepsilon\,\omega(r_j),\qquad  w_i := (k_i-u_j)_+,
\end{equation*}
and
\begin{equation*}     
A_i=\frac{|B^{i}\cap \{u_j \leq k_i\}|}{|B^{i}|} = \frac{|B^{i}\cap \{w_i > 0\}|}{|B^{i}|}.
\end{equation*}
The Caccioppoli inequality in~\eqref{cacio1} now yields 
\begin{eqnarray}\label{cacio 2}
\nonumber && \int_{B^i}\int_{B^i}K(x,y)|w_i (x)\phi_i(x)-w_{i}(y)\phi_i(y)|^p \,{\rm d}x{\rm d}y\\
 &&\qquad \qquad  \leq c\int_{B^i}\int_{B^i}K(x,y)
 (\max\{w_{i}(x),w_i(y)\})^p |\phi_i(x)-\phi_i(y)|^p \,{\rm d}x{\rm d}y\\
&&\qquad \qquad \quad+\,c \,\int_{B^i}w_{i}(x)\phi_i^p(x)\,{\rm d}x \left(\sup_{y\,\in\, \tilde B^i}\int_{\mathds{R}^n\setminus B^i }K(x,y)w_{i}^{p-1}(x)\,{\rm d}x \right).
\nonumber
\end{eqnarray}
We can estimate the term on the left below as
\begin{eqnarray}\label{15star}
A_{i+1}^{\frac{p}{p^*}}(k_i-k_{i+1})^p &=& \frac{1}{|B^{i+1}|^{\frac{p}{p^*}}}\left(\int_{B^{i+1}\,\cap\, \{u_j\leq k_{i+1}\}}(k_i-k_{i+1})^{p^*}\phi_i^{p^*}(x)\,{\rm d}x\right)^{\frac{p}{p^*}} \nonumber\\[1ex]
&\leq & \frac{1}{|B^{i+1}|^{\frac{p}{p^*}}}\left(\int_{B^{i}}w_i^{p^*}(x)\phi_i^{p^*}(x)\,{\rm d}x\right)^{\frac{p}{p^*}} \nonumber\\[1ex]
&\leq&c \, r_{j+1}^{sp-n}\int_{B^i}\int_{B^i}K(x,y)|w_i(x)\phi_i(x)-w_i(y)\phi_i(y)|^p  \,{\rm d}x{\rm d}y.
\end{eqnarray}
Recalling that $|D\phi_i |\leq c 2^{i} r_{j+1}^{-1}$, the first term on the right in~\eqref{cacio 2} can be treated as follows, 
\begin{eqnarray}
&& \notag r_{j+1}^{sp} \int_{B^{i}}\int_{B^i}K(x,y) (\max\{w_{i}(x),w_i(y)\})^p |\phi_i(x)-\phi_i(y)|^p\,{\rm d}x{\rm d}y \\
&& \notag \qquad \leq  c \,2^{ip} r_{j+1}^{sp} r_{j+1}^{-p}\,k_i^p \int_{B^i\,\cap\,\{u_j\leq k_i\}}\int_{B^i}\frac{1}{|x-y|^{-p+n+sp}}\,{\rm d}y{\rm d}x\\[1ex]
&& \label{est on right 1} \qquad  \leq c\,  2^{ip} \left[ \eps \omega(r_j) \right]^p\,|B^i\cap \{u_j\leq k_i\}|.
\end{eqnarray}
Moreover, 
\begin{equation} \label{est on right 2}
\int_{B^i} w_i(x)\phi_i^p(x)\,{\rm d}x
\, \leq \, 
c \left[ \eps \omega(r_j) \right] |B^i\cap \{u_j\leq k_i\}|
\end{equation}
holds. To tackle the third integral in~\eqref{cacio 2}, we first have
\begin{equation}\label{taaa}
r_{j+1}^{sp}\left(\sup_{y\,\in\, \tilde B^i}\int_{\mathds{R}^n\setminus B^i }K(x,y)w_{i}^{p-1}(x)\,{\rm d}x \right) \, \leq \, c 2^{i(n+sp)} \left[{\rm Tail}(w_i;x_0,r_{j+1})\right]^{p-1}, 
\end{equation}
using
\begin{equation*}
\inf_{y \in \tilde B^i} |y-x| \, \geq \,  |x_0-x| \inf_{y \in \tilde B^i}\frac{|y-x|}{|x_0-x|}
\, \geq \, 2^{-i-1} |x-x_0|
\end{equation*}
for all $x \in \mathds{R}^n\setminus B^i $ and the fact that
$$
B_{r_{j+1}}\equiv B_{j+1}\subset B^i \quad \Rightarrow \quad \mathds{R}^n\setminus B^i \subset \mathds{R}^n\setminus B_{j+1}. 
$$
Recalling~\eqref{step1} and the facts that $w_i \leq 2\eps \omega(r_j)$ in $B_{j}$ and $w_i \leq |u_j| + 2\eps \omega(r_j)$ in $\mathds{R}^n$, we further get
\begin{eqnarray*} 
\notag && \left[{\rm Tail}(w_i;x_0,r_{j+1})\right]^{p-1}
\\[1ex] \notag && \qquad \leq cr_{j+1}^{sp}\int_{B_{j}\setminus B_{j+1}} w_i^{p-1}(x) |x-x_0|^{-n-sp} \, {\rm d}x +  c \left(\frac{r_{j+1}}{r_j} \right)^{sp} \left[{\rm Tail}(w_i;x_0,r_{j})\right]^{p-1}
\\[1ex] \notag && \qquad \leq c \eps^{p-1} \omega(r_j)^{p-1}  + c \sigma^{sp}\left[{\rm Tail}(u_j;x_0,r_{j})\right]^{p-1}
\\[1ex] \notag && \qquad \leq c\left(1 + \frac{\sigma^{sp-\alpha(p-1)}}{\eps^{p-1}} \right) \left [\eps \omega(r_j)\right]^{p-1}
\\[1ex] \notag && \qquad \leq c \left [\eps \omega(r_j)\right]^{p-1},
\label{eq:} 
\end{eqnarray*}
by the very definition of $\eps$.  Combining the estimates above, we deduce that
\begin{equation} \label{est on right 3}
r_{j+1}^{sp}  \left(\sup_{y\,\in\, \tilde B^i}\int_{\mathds{R}^n\setminus B^i }K(x,y)w_{i}^{p-1}(x)\,{\rm d}x\right)  \, \leq  \,
c \, 2^{i(n+sp)}  \left [\eps \omega(r_j)\right]^{p-1}.
\end{equation}

Putting together~\eqref{cacio 2},~\eqref{15star}~\eqref{est on right 1},~\eqref{est on right 2} and~\eqref{est on right 3}, we arrive at 
$$
A_{i+1}^{\frac{p}{p^\ast}}(k_i-k_{i+1})^p \, \leq  \, c\, 2^{i(n+sp+p)}  \left [\eps \omega(r_j)\right]^{p}  A_i,
$$
which yields
\begin{equation*}
A_{i+1}\, \leq \, c\, 2^{i\,[n+(2+s)p]p^\ast/p} A_i^{1+\beta}
\end{equation*}
with $\beta :=sp/(n-sp)$ by the definition of $k_i$'s. Now, we recall that  if we prove the following estimate on $A_0$, 
\begin{equation}\label{A0}
A_0=\frac{|\tilde{B}\cap \{u_j\leq 2\varepsilon \omega(r_j)\}|}{|\tilde{B}|}\leq c^{-1/\beta}2^{-[n+(2+s)p]p^\ast/[p\beta^2]} =: \nu^\ast,
\end{equation}
then we can deduce that
$$
A_i\to 0 \quad \text{as} \quad i\to \infty.
$$
Indeed, the condition~\eqref{A0} we can guarantee by~\eqref{MZ} choosing 
\[
\sigma = \min\{1/4, \, \exp(- c_{\log}/\nu^\ast)\},
\]
which then depends only on $n,p,s,\lambda,\Lambda$ and the difference of $sp/(p-1)$ and $\alpha$. In other words, we have shown that
\[
\osc_{B_{j+1}} u \leq (1-\eps) \omega(r_j)=(1-\eps)\left(\frac{r_j}{r_{j+1}}\right)^\alpha \omega(r_{j+1})=(1-\eps)\sigma^{-\alpha} \omega(r_{j+1}).
\]
Taking finally $\alpha \in \left(0,\frac{sp}{p-1}\right)$ small enough satisfying
\[
\sigma^\alpha \geq 1-\eps =  1 - \sigma^{\frac{sp}{p-1}-\alpha},
\]
then, clearly, $\alpha$ depends only on $n,p,s,\lambda,\Lambda$ and 
\[
\osc_{B_{j+1}} u \leq \omega(r_{j+1})
\]
holds, proving the induction step and finishing the proof. 
\end{proof}

\noindent
\\ {\bf Acknowledgements.}\
The authors have been supported by the \href{http://prmat.math.unipr.it/~rivista/eventi/2010/ERC-VP/}{ERC grant 207573 ``Vectorial Problems''}. The second author  has also been supported by Academy of Finland project ``Regularity theory for nonlinear parabolic partial differential equations", and the third author by PRIN 2010-11 ``Calcolo delle Variazioni''. The first and the third authors are members of Gruppo Nazionale per l'Analisi Matematica, la Probabilit\`a e le loro Applicazioni (GNAMPA) of Istituto Nazionale di Alta Matematica ``F.~Severi'' (INdAM), whose support is acknowledged.

 We would like to thank Lorenzo Brasco and Enea Parini for careful reading of a preliminary version of the manuscript. Finally, we would like to thank the  referees for their  useful suggestions, which allowed to improve the manuscript.

\vspace{3mm}
\end{document}